\documentclass[12pt]{amsart}
\usepackage{amssymb,amscd,amsthm,hyperref,cleveref,mathtools}
\usepackage[shortlabels]{enumitem}
\usepackage[margin=0.8in]{geometry}
\usepackage{colonequals}
\usepackage{url}

\usepackage{colortbl}
\definecolor{lightgray}{rgb}{0.9,0.9,0.9}

\usepackage[all,cmtip]{xy}

\def\co{\colon\thinspace}

\newcommand\GG{\mathbb{G}}
\newcommand\PP{{\mathbb{P}}}

\newcommand\LL{{\mathcal{L}}}
\newcommand\MM{{\mathcal{M}}}
\newcommand\WW{{\mathcal{W}}}
\newcommand\RR{\mathbb{R}}
\newcommand\ZZ{\mathbb{Z}}

\def\coeq{\colonequals}

\DeclareMathOperator*{\colim}{colim}

\DeclareMathOperator{\Aut}{Aut}
\DeclareMathOperator{\cft}{cft}
\DeclareMathOperator{\Der}{Der}
\DeclareMathOperator{\Sing}{Sing}
\DeclareMathOperator{\dR}{dR}
\DeclareMathOperator{\Ho}{\mathbf{Ho}}
\DeclareMathOperator{\Lie}{Lie}
\DeclareMathOperator{\DGCA}{\mathbf{DGCA}}
\DeclareMathOperator{\DGCC}{\mathbf{DGCC}}
\DeclareMathOperator{\DGLA}{\mathbf{DGLA}}
\DeclareMathOperator{\fft}{fft}
\DeclareMathOperator{\Nil}{\mathbf{Nil}}
\DeclareMathOperator{\sSet}{\mathbf{sSet}}
\DeclareMathOperator{\SU}{SU}
\DeclareMathOperator{\Top}{\mathbf{Top}}
\DeclareMathOperator{\op}{op}
\DeclareMathOperator{\PL}{PL}
\DeclareMathOperator{\qis}{q-is}

\DeclareMathOperator{\Map}{Map}
\DeclareMathOperator{\MC}{MC}
\DeclareMathOperator{\Hom}{Hom}
\DeclareMathOperator{\SR}{SR}

\newcommand{\BB}{\mathbb{B}}
\newcommand{\CC}{\mathbb{C}}

\newcommand{\QQ}{\mathbb{Q}}

\newcommand{\abs}[1]{\lvert #1 \rvert}

\newcommand{\dslash}{/\!\!/}

\newcommand{\h}{\mathfrak{h}}

\newcommand{\Li}{L_\infty}
\newcommand{\LiAlg}{\mathbf{\Li}\text{-}\mathbf{Alg}}

\newcommand{\cO}{\mathcal{O}}

\newcommand{\Set}{\mathbf{Set}}

\newcommand{\mc}[1]{\mathcal{#1}}

\newtheorem{thm}{Theorem}
\newtheorem{crl}[thm]{Corollary}

\theoremstyle{definition}

\newtheorem{example}{Example}[section]

\theoremstyle{remark}

\newtheorem{remark}{Remark}[section]

\author[A. A. Voronov]{Alexander A. Voronov}
\address{School of Mathematics,
  University of Minnesota,
  Minneapolis, MN 55455 and
 Kavli IPMU (WPI), UTIAS, The University of Tokyo, Kashiwa, Chiba 277-8583, Japan
 }
\thanks{Research is supported in part by World Premier International Research Center Initiative (WPI Initiative), MEXT, Japan, 
and Travel grant \#585720 from the Simons Foundation. }

\title
{Rational Homotopy Theory}

\begin{document}
\begin{abstract}
This is a survey of Rational Homotopy Theory, intended for a Mathematical Physics readership.
\end{abstract}

\maketitle

\tableofcontents

\section*{Introduction}
Rational Homotopy Theory (RHT) is a story of success in Algebraic
Topology. It sacrifices the all-time favorite homotopy equivalence
relation on topological spaces for the weaker rational homotopy
equivalence, ignoring all the torsion in homotopy and homology groups,
and provides a complete algebraic invariant of the rational
equivalence class of a good enough space $X$, the isomorphism class of
a differential graded (dg-) commutative algebra (DGCA) $M(X)$, called
the Sullivan minimal model of $X$: two such spaces are rationally
equivalent if and only if their Sullivan minimal models are
isomorphic. Moreover, with suitable model-category structures, the
homotopy category of good enough DGCAs becomes equivalent to the
rational homotopy category of good enough spaces. Another feature of
the DGCA $M(X)$ is that it is rather easily computable, making the
invariant $M(X)$ an effective tool of topology. A parallel,
historically earlier algebraic model of a space $X$ is given by a
differential graded Lie algebra (DGLA) $Q(X)$, called the Quillen
minimal model, and provides practically the same features. Koszul
duality between the commutative and Lie operads, \cite{GK}, brings the
two models together.

Perhaps, one of the most spectacular mathematical applications of RHT
is the result of D.~Sullivan and M.~Vigu\'{e}-Poirrier
\cite{Sullivan.Vigue-Poirrier.1976, Felix.Oprea.Tanre} which states that a
sufficiently general simply connected Riemannian manifold has
infinitely many closed geodesics. RHT has made a strong appearance in
Mathematical Physics via H.~Sati's \emph{Hypothesis H}, which states,
roughly speaking, that the Sullivan minimal model of the 4-sphere
$S^4$ captures the dynamics of fields in M-theory in the low-energy
(UV) limit, also known as 11d supergravity
\cite{Sati.2013}. Hypothesis H progresses to the dimensional
reductions of M-theory with its dynamics described by the iterated
cyclic loop space $\mathcal{L}_c^k S^4$ of the 4-sphere \cite{FSS17,
  FSS19b, FSS-WZW, GS21, SV1, SV2}.

In this survey, we review the main ideas and results of RHT. RHT
cannot complain about the lack of surveys, see
\cite{Hess,Berglund.RHT,Menichi,Moerdijk,Wang,Holstein}. What
distinguishes this survey from the variety of previous ones is that
it, first of all, keeps an eye on Mathematical Physics. A novel
feature of this survey is that more recent developments which engage
$\Li$-algebras in lieu of DGLAs that are free as graded Lie algebras
are put forward as Quillen minimal models. Another novelty, which may
get a rise out of seasoned rational homotopy theorists, is the use of
the graded symmetric algebra notation $S(V)$ in lieu of $\wedge V$, as
well as swapping the standard abbreviation CDGA for DGCA. The reason
for the first deviation from the standard convention is that in the
category of graded vector spaces, the graded exterior algebra is not
free as a graded commutative algebra and not isomorphic to the graded
symmetric algebra. The reason for the second deviation from the
standards is that the placement of modifiers may be commutative but
not necessarily associative: D(GC)A $\ne$ (DG)CA $=$ C(DG)A. Whatever
abbreviation the practitioners use, they always mean D(GC)A in this
context.

Some novelties in exposition do not mean novelties in the content:
this article presents only known results. The only possible exception,
which may appear for the first time in written form, is \Cref{CS} on
the Chern-Simons interpretation of the ``phantom'' 3-form $g_3$ on the
sphere $S^2$. And even that is not due to the author of the survey.

\textbf{Disclaimer}: This survey, being limited in length, leaves
aside a number of remarkable developments of RHT, such as, for
example, those related to mapping spaces \cite{Sch.St.2012,
  Buijs.Murillo.2013, Berglund.rhtoms.2015, Wierstra.2019, BFMT.2020,
  RNV.2020} or disconnected spaces
\cite{Lazarev.Markl.2015}. Likewise, we do not pretend to cover all
the applications of RHT to physics. The choice of topics obviously
reflects the bias of the author.

\textbf{Conventions}: When we use the term ``space,'' we will mean a
topological space. Moreover, we will assume all spaces to be
\emph{path-connected} unless stated otherwise. All manifolds will be
smooth ($C^\infty$) by default. In graded algebra, $\abs{a}$ will
denote the degree of an element $a$.

\smallskip
{\bf Acknowledgments}: I am grateful to my collaborator Hisham Sati,
who introduced me to the beautiful interaction of rational homotopy
with physics, Alexey Bondal and Mikhail Kapranov, with whom I had
numerous fruitful discussions at Kavli IPMU, and Eli Schlossberg, who
shared his notes on rational homotopy and deformation theories with
me.

\section{Rational homotopy category of spaces}
\label{RHC}

A continuous map $f: X \to Y$ between topological spaces is called a
\emph{rational $($homotopy$)$ equivalence} if it induces an
isomorphism on rational homology: $f_*: H_\bullet (X; \QQ) \to
H_\bullet (Y; \QQ)$. (Compare this to a stronger notion of weak
equivalence: to compare apples to apples and oranges to oranges, let
us say a continuous map $f: X \to Y$ of simply connected spaces is
called a \emph{weak $($homotopy$)$ equivalence} if it induces an
isomorphism $f_*: H_\bullet (X; \ZZ) \to H_\bullet (Y; \ZZ)$ on
integral homology.) We say that two spaces $X$ and $Y$ are
\emph{rationally equivalent} if they may be connected by a sequence of
rational equivalences $X \leftarrow X_1 \to \dots \leftarrow X_n
\rightarrow Y$. The equivalence class of a space $X$ up to rational
equivalence is called the \emph{rational homotopy type of} $X$. If we
add formal inverses of rational equivalences to the category of
spaces, i.e., we localize the category with respect to rational
equivalences, we obtain the \emph{rational homotopy category of
spaces}.

\begin{example}
The unique map $\RR \PP^2 \to \{*\}$ from the real projective plane to
a singleton is a rational equivalence, because the homology groups
$H_n(\RR \PP^2; \ZZ)$, $n > 0$, are all torsion.
\end{example}

\begin{example}
Let $S^{2n-1}$, $n \ge 1$, be an odd-dimensional sphere and $K(\ZZ,
2n-1)$ be the Eilenberg-MacLane space. Let $S^{2n-1} \to K(\ZZ, 2n-1)$
be a map realizing a generator of $\pi_{2n-1} (K(\ZZ, 2n-1)) =
\ZZ$. This map is a rational equivalence, because it induces an
isomorphism on rational homotopy groups:
\[
\pi_{i} (S^{2n-1}) \otimes \QQ \cong \pi_{i} (K(\ZZ, 2n-1))
\otimes \QQ =
\begin{cases}
  \QQ, & \text{if $i = 2n-1$,}\\
  0, & \text{otherwise},
  \end{cases}
\]
cf.\ the remark at the end of the second paragraph of
\Cref{model-cat}. Thus, $S^{2n-1}$ is a \emph{rational Eilenberg-MacLane
space}. The same argument does not work for even-dimensional spheres
$S^{2n}$, because $S^{2n}$ has an extra rational homotopy group in
degree $4n-1$:
\[
\pi_{i} (S^{2n}) \otimes \QQ \cong 
\begin{cases}
  \QQ, & \text{if $i = 2n$ or $4n-1$,}\\
  0, & \text{otherwise}.
  \end{cases}
\]
\end{example}

\section{Differential graded commutative algebras}

RHT assigns an algebraic model, namely a DGCA, see below, to a
space. Such an algebra is supposed to be large enough to carry all the
information about the rational homotopy type of the space. The theory
works further to replace this large DGCA with a much smaller model,
called a minimal Sullivan model. The prototypical example of an
excessively large model is the de Rham algebra $\Omega^\bullet(X)$ of
a manifold $X$. In this case, the ground field has to be
extended to $\RR$, and we are dealing with the \emph{real}, rather
than rational, \emph{homotopy type} of $X$.

A \emph{differential graded commutative algebra $($DGCA$)$} over a
field $k$ is a graded associative $k$-algebra $A = \bigoplus_{n \in
  \ZZ} A^n$ with a multiplication
\begin{align*}
  A \otimes A & \to A,\\
  a \otimes b & \mapsto ab,
\end{align*}
which respects the grading: $A^m \otimes A^n \to A^{m+n}$, is
\emph{graded commutative}:
\[
ba = (-1)^{\abs{a}\abs{b}} ab,
\]
has a unit $ 1 \in A$, and has a differential $d: A \to A$, a
$k$-linear operator of degree 1 such that $d^2 = 0$ and
it is a (graded) \emph{derivation} of $A$:
\[
d(ab) = (da)b + (-1)^{\abs{a}} a (db).
\]
The derivation property, also known as the \emph{Leibniz rule},
implies that the differential kills the constants: $d(1) = 0$. A DGCA
\emph{homomorphism} is a linear map $A \to B$ which respects the DGCA
structure: the multiplication, unit, grading, and differential. In
this survey, a DGCA will by default be nonnegatively graded ($A =
\bigoplus_{n \ge 0} A^n$) over a field $k$ of characteristic 0.

A DGCA $A$ is called \emph{$($homologically$)$ connected} if $A^0 = k$
(respectively, $H^0(A) = k$). It is called \emph{$($homologically$)$
simply connected} if $A^0 = k$ and $A^1 = 0$, (respectively, $H^0(A) =
k$ and $H^1(A) = 0$).

A DGCA $(S(V), d)$ is called \emph{semifree}, if it is based on the
free graded commutative algebra $S(V) = \bigoplus_{n \ge 0} S^n(V)$ on
a positively graded vector space $V = \bigoplus_{n>0} V^n$.

\section{Minimal Sullivan models of DGCAs}

RHT provides a minimal Sullivan model for each good enough DGCA. A
\emph{Sullivan algebra} is a semifree DGCA $(S(V), d)$ such that the
differential $d$ satisfies the following nilpotency condition, also
known as the \emph{Sullivan condition}: there exists a filtration of
the generating space $V$ by graded subspaces:
\[
0 = V(0) \subset V(1) \subset \dots
\]
such that $V = \bigcup_{n \ge 0} V(n)$ and $d (V(n)) \subset
S(V(n-1))$ for all $n \ge 1$. A Sullivan algebra is automatically
connected.

A semifree DGCA $(S(V), d)$ is called \emph{minimal} if $d(V) \subset
S^{\ge 2} (V)$.

Every simply connected minimal DGCA $S(V)$ is necessarily Sullivan:
the filtration associated with the grading does the job. Indeed, we
have $V^1 = 0$ and for
\[
V(n) \coeq \bigoplus_{2\le i \le n} V^i, \qquad n \ge 0,
\]
the Sullivan nilpotency condition follows from $d (V^n) \subset S^{\ge
  2}(V)^{n+1} \subset S(V(n-1))$ for all $n \ge 1$.

A DGCA \emph{quasi-isomorphism} is a DGCA homomorphism $A \to B$ which
induces and isomorphism on the cohomology of the differentials
$H^\bullet (A, d_A) \xrightarrow{\sim} H^\bullet (B, d_B)$. In this
case we write $A \xrightarrow{\qis} B$ and say that $A$ (or $A
\xrightarrow{\qis} B$, to be more precise) is a \emph{model} of
$B$. Sullivan and minimal Sullivan models play an important role in
RHT for a variety of reasons, mostly related to the fact that they are
cofibrant DGCAs in a suitable model-category structure, see
\Cref{model-cat}.

\begin{thm}[{\cite[Theorems 14.11-14.12]{RHT}}]
  \label{unique}
  \begin{enumerate}[$(1)$]
  \item
    Every quasi-isomorphism between minimal Sullivan algebras is an
    isomorphism.
  \item
    Every homologically connected DGCA has a minimal Sullivan model,
    unique up to isomorphism.
  \end{enumerate}
\end{thm}

We conclude this section with a structure theorem for Sullivan
algebras.

\begin{thm}[{\cite[Theorem 14.9]{RHT}}]
  \label{product}
  Every connected Sullivan algebra is isomorphic to the tensor product
  of a minimal Sullivan algebra and a semifree contractible algebra
  $S(U \oplus dU)$, where $U$ is a positively graded vector space.
\end{thm}

\section{Sullivan minimal models of spaces}

Every manifold $X$ has a DGCA associated to it, the de Rham
algebra $(\Omega^\bullet(X), d_{\dR})$. This is quite a familiar a
DGCA over the reals $\RR$ whose cohomology is the de Rham cohomology
$H^\bullet(X; \RR)$. Sullivan \cite{Sullivan} has generalized this
construction to work for an arbitrary topological space $X$ and
produce a DGCA $A_{\PL}(X)$ over the rationals $\QQ$ whose cohomology
is isomorphic to the rational cohomology of $X$:
\[
H^\bullet (A_{\PL} (X)) \cong H^\bullet (X; \QQ).
\]
The DGCA $A_{\PL} (X)$, called the polynomial de Rham algebra or the
algebra of polynomial differential forms on $X$, is defined in
\Cref{ratlmodels}. By definition, a Sullivan minimal model $M(X)$ of a
path-connected space $X$ is a minimal Sullivan model of $A_{\PL}
(X)$. Since $M(X)$ is defined uniquely up to isomorphism, it is
usually called \emph{the} Sullivan minimal model of $X$. One has
\begin{equation}
  \label{cohomology}
H^\bullet (M (X)) \cong H^\bullet (X; \QQ)
\end{equation}
and 
\begin{equation}
  \label{homotopy}
M (X)^{> 0}/ (M (X)^{> 0})^2 \cong \Hom_{\ZZ} (\pi_\bullet (X), \QQ).
\end{equation}
This implies that $M(X) \cong S(\Hom_{\ZZ} (\pi_\bullet (X), \QQ))$.

\subsection{Real algebraic models of manifolds}

If $X$ is a connected manifold, then $\Omega^\bullet(X)$ is
quasi-isomorphic to $A_{\PL} (X) \otimes_\QQ \RR$ and defines the real
homotopy type of $X$, \cite[Section 2.4.1]{Felix.Oprea.Tanre}. The
\emph{Sullivan minimal model $M_\RR(X)$ of $X$ over $\RR$} is then the
minimal Sullivan model of $\Omega^\bullet(X)$. Let us compute the real
homotopy type, i.e., the Sullivan minimal model over the reals, of the
sphere $S^n$ for every $n \ge 1$. The Sullivan minimal model of $S^n$
over $\QQ$ will be exactly the same, except that the ground field will
be changed to $\QQ$.

\begin{example}[The real homotopy type of an odd-dimensional
    sphere] We claim that the polynomial algebra $\RR[g_n] \cong \RR
  \oplus \RR g_n$ in a variable $g_n$ of degree $n$ with differential
  determined by $d g_n = 0$ is a Sullivan minimal model of
  $S^n$. Indeed, choose a differential $n$-form $\omega_n$ on $S^n$
  generating its cohomology $H^n (S^n; \RR)$ and define an algebra
  homomorphism
\begin{align*}
  \RR[g_n] & \to \Omega^\bullet (S^n),\\
  g_n & \mapsto \omega_n,
\end{align*}
  which obviously respects the grading and differentials. Moreover, it
  is clear it is a quasi-iso\-mor\-phism. $\RR[g_n]$ is trivially
  Sullivan and minimal. We conclude that $M_\RR (S^n) = \RR[g_n]$ with
  $\abs{g_n} = n$ and $d = 0$ for $n$ odd.
\end{example}

\begin{example}[The real and rational homotopy type of an even-dimensional
    sphere]
  \label{even}
  If $n \ge 2$ is even, then arguing similarly, we observe that the
  cohomology algebra $(\RR[g_n]/(g_n^2), 0)$, admits a
  quasi-isomorphism to the de Rham algebra $\Omega^\bullet (S^n)$. The
  problem is that $(\RR[g_n]/(g_n^2), 0)$ is not Sullivan, as it is
  not even semifree. If we drop the relation $g_n^2 = 0$ to gain
  freeness, we will acquire cohomology in each degree which is a
  multiple of $n$. To kill the cohomology class $g_n^2$, we introduce
  a new generator $g_{2n-1}$ of degree $2n-1$ with $dg_{2n-1} = g_n^2$
  and get the semifree resolution $\RR[g_n, g_{2n-1}]
  \xrightarrow{\qis} \RR[g_n]/(g_n^2)$ with $d g_n = 0$ and $d
  g_{2n-1} = g^2_n$ of $(\RR[g_n]/(g_n^2), 0)$. This semifree
  resolution is simply connected and minimal and thereby
  Sullivan. Thus, $M_\RR (S^n) = \RR[g_n, g_{2n-1}]$ with $\abs{g_n} =
  n$, $\abs{g_{2n-1}} = 2n-1$, and $d g_n = 0$, $d g_{2n-1} = g_n^2$
  for $n$ odd. Note that we could have set $dg_{2n-1} = c g^2_n$ for
  any nonzero $c \in \RR$. The resulting DGCA would be isomorphic to
  the one with $c=1$. The choice of $c = -1/2$ looks more like the
  Maurer-Cartan equation, $dg_{2n-1} + \tfrac{1}{2} g^2_n = 0$, and is
  particularly common in physics, see \cite{Sati.2013}.
\end{example}

Note that in both examples, the cohomology algebra is quasi-isomorphic
to the de Rham algebra. In such cases, we say that the manifold and
its de Rham algebra are \emph{formal}. For a general DGCA $A$ or space
$X$, we say it is \emph{formal}, if $A$ or, respectively,
$A_{\PL}(X)$, is quasi-isomorphic to its cohomology DGCA (with the
zero differential). Another famous class of formal spaces are compact
K\"{a}hler manifolds, as per the Deligne-Griffiths-Morgan-Sullivan
theorem \cite{DGMS}.

\begin{example}[Realization of $g_3$ for $S^2$ as a Chern-Simons form]
  \label{CS}
For an even number $n \ge 2$, one may wonder where the element
$g_{2n-1}$ for the sphere $S^n$ lives. It is certainly not there in
the cohomology or even the de Rham algebra of $S^n$, as both vanish in
degree $2n-1$.  M.~Kapranov has suggested that $g_3$ for $S^2$ may be
realized as a Chern-Simons form on $X = \CC \PP^\infty \setminus \CC
\PP^{\infty-2}$, which is another model of $S^2$: $X = \{ [z_0: z_1 :
  z_2 : \dots ] \in \CC \PP^\infty \mid z_0 \ne 0 \text{ or } z_1 \ne
0\} $ deformation retracts to the subspace $0 = z_2 = z_3 = \dots$,
which is homeomorphic to $\{ [z_0: z_1] \mid z_0 \ne 0 \text{ or } z_1
\ne 0\} = \CC \PP^1 \cong S^2 $. Let $\omega_2$ be the
\emph{curvature} form of the Chern connection associated to the
standard Hermitian metric on $\cO(1)$, the dual tautological line
bundle on $\CC \PP^\infty$. The cohomology class of $\tfrac{i}{2 \pi}
\omega_2$ is the first Chern class $c_1(\cO(1))$ of the line bundle
$\cO(1)$. The cohomology class of $\omega_2$ freely generates the real
cohomology $\RR[\omega_2]$ of $\CC \PP^\infty$ and no power of
$\omega_2$ vanishes on the dense open $X \subset \CC
\PP^\infty$. However, the cohomology class of $\omega_2 \wedge
\omega_2$ in $H^4(X; \RR) \cong H^4(S^2; \RR) = 0$ vanishes and hence
there exists a 3-form $\omega_3$ on $X$ such that $d \omega_3 =
\omega_2 \wedge \omega_2$. In this sense, $\omega_3$ is the
rudimentary, abelian version of a Chern-Simons form, being associated
to a connection on a complex line bundle, rather than an
$\SU(n)$-bundle for $n \ge 2$. Note that the whole DGCA $M(S^2) =
\RR[g_2, g_3]$ may be realized as the sub-DGCA $\RR[\omega_2,
  \omega_3]$ of the de Rham algebra $\Omega^\bullet(X)$, which is
another algebraic model of $S^2 \sim X$:
\[
M(S^2) = \RR[g_2, g_3] \cong \RR[\omega_2, \omega_3]
\xhookrightarrow[]{\qis} \Omega^\bullet(X).
\]
\end{example}

\subsection{Rational algebraic models of spaces}
\label{ratlmodels}

Here we define the polynomial de Rham algebra $A_{\PL} (X)$ for a
topological space $X$. We will do it by considering the simplicial set
$\Sing (X)$ of singular simplices of $X$ and gluing the algebra out of
polynomial differential forms on each singular simplex.

The \emph{polynomial de Rham algebra} or the \emph{algebra of
polynomial differential forms on the standard $n$-simplex}
\[
\Delta^n \coeq \left\{ (t_0,\dots,t_n)\in\mathbf{R}^{n+1} ~\Bigg|~ \sum_{i
  = 0}^n t_i = 1 \text{ and } t_i \ge 0 \text{ for } i = 0, \ldots,
n\right\}
\]
is the DGCA
\[
\Omega^\bullet([n]) \coeq \QQ[t_0, ..., t_n, dt_0,
  ..., dt_n]/(\sum t_i - 1, \sum dt_i).
\]
Varying $n$, we get a simplicial DGCA,
$\Omega^\bullet$. For any simplicial set $K$, define the
\emph{polynomial de Rham algebra of $K$} as
\[
A_{\PL}(K) \coeq \Hom_{\sSet} (K, \Omega^\bullet).
\]
Here $\Hom_{\sSet}$ stands for the set of morphisms of simplicial
sets, i.e., maps of sets $K([n]) \to
\Omega^\bullet([n])$ for each $n \ge 0$ respecting the
face and degeneracy maps, or equivalently, defining natural
transformations between the corresponding functors $K \co \Delta \to
\Set$ and $\Omega^\bullet \co \Delta \to \Set$ on the
simplex category $\Delta$, where $\Set$ is the category of sets.
The structure of a DGCA on $A_{\PL}(K)$ comes from that on each
$\Omega^\bullet([n])$. Finally, using the simplicial set
$\Sing (X)$ of singular simplices of a space, we define the
\emph{polynomial de Rham algebra of a space $X$}:
\[
A_{\PL}(X) \coeq \Hom_{\sSet} (\Sing (X), \Omega^\bullet).
\]
For a path-connected space $X$, this DGCA is connected, called a
\emph{$($rational$)$ model of $X$}, and its minimal Sullivan model is
called the \emph{Sullivan minimal model $M(X)$ of $X$}. The minimal
Sullivan model of a space is unique up to isomorphism. For good enough
spaces, (the isomorphism class of) the Sullivan minimal model
determines the space up to rational equivalence, that is to say,
\[
X \sim_\QQ Y \qquad \text{if and only if} \qquad M(X) \cong M(Y).
\]
Good enough spaces include spaces of finite $\QQ$-type which are
simply connected, or, more generally, nilpotent, see definitions in
\Cref{MainThm}. For example, if the Sullivan minimal model of a simply
connected space is $\QQ[g_4, g_7]$ with $\abs{g_4} = 4$, $\abs{g_7} =
7$, and $dg_4 = 0$, $dg_7 = -\tfrac{1}{2} g_4^2$, we may rest assured
that this space is rationally equivalent to the sphere $S^4$, see
\Cref{even}.

\subsection{Realizations}

In the previous section, we associated a DGCA to every topological
space. We did that in two steps, going first from spaces to simplicial
sets and then to DGCAs.
Here we will discuss a construction that produces a topological space
for a given DGCA, also in two steps. Let $A$ be a nonnegatively graded
DGCA. Its \emph{simplicial realization} is
the simplicial set
\[
K_{\SR}(A) \coeq \Hom_{\DGCA}(A, \Omega^\bullet),
\]
where $\DGCA$ is the category of nonnegatively graded DGCAs.
This construction actually defines a (contravariant) functor $K_{\SR}
\co \DGCA \to \sSet$, which turns out to be left adjoint for the
(contravariant) polynomial de Rham functor $A_{\PL} \co \sSet \to
\DGCA$ of the previous section. More about this is in
\Cref{model-cat}, in which we discuss model categories.

To pass to topological spaces, we can further apply the functor of
\emph{geometric realization}:
\[
K \mapsto \abs{K} \coeq \colim \left(\xymatrix{
  \coprod_{[n] \to [m]} K([m])\times \Delta^n
  \ar@<.5ex>[r]
\ar@<-.5ex>[r]  &
\coprod_{[n]} K([n])\times \Delta^n
}\right).
\]
In other words, we take the disjoint union $\coprod_{n} K([n])\times
\Delta^n$ of copies of the standard simplex $\Delta^n$, one for each
$n$-simplex of $K$, for all $n \ge 0$, and glue their faces as
prescribed by the face maps of $K$ and degenerate them as prescribed
by the degeneracy maps of $K$. This functor turns out to be left
adjoint for the singular simplex functor $\Sing$.

Finally, a \emph{spatial realization} of a DGCA $A$ is the geometric
realization of the simplicial realization:
\[
X_{\SR}(A) \coeq \abs{K_{\SR}(A)}.
\]
Unfortunately, this functor is not the left adjoint of the polynomial
de Rham functor $A_{\PL}(X)$: the contravariance of $K_{\SR}$ messes
things up, and we end up composing two covariant functors, to be on
the safe side:
\begin{equation*}
 \xymatrixrowsep{1pc}
  \xymatrix{
    {\DGCA} \ar[r]^{{K_{\SR}}}
    & {\sSet}^{\op}  \ar[r]^{{\abs{-}}^{\op}}
    & {\Top}^{\op},\\
  A  \ar@{|->}[r]  & K_{\SR}(A) \ar@{|->}[r] & \abs{K_{\SR}(A)},
  }
\end{equation*}
one which, $K_{\SR}$, is a left adjoint functor and the other one,
$\abs{-}^{\op}$, is a right adjoint, being the opposite of a left
adjoint functor. The notation in the first row is spelled out below in
\Cref{model-cat}. However, the homotopy theory in $\sSet$ and $\Top$
is just the same, as the functors
\[
      \xymatrix{
\abs{-} \co \sSet \ar@/^/[r] & \Top : \Sing \ar@/^/[l]
      }
\]
constitute a Quillen equivalence, and every topologist feels as
comfortable in $\sSet$ as in $\Top$, perhaps even more so. The details
are also explained in \Cref{model-cat}.

\section{Rationalization}

A rationalization may be viewed as a rational topological model of a
space, categorically dual to the Sullivan models considered above. A
path-connected space is called \emph{rational} if the $\ZZ$-module
$H_n(X; \ZZ)$ is a $\QQ$-vector space for each $n \ge 1$. (This means
unique divisibility of each element by natural numbers.) A
\emph{rationalization} of a path-connected space $X$ is a rational
space $X_\QQ$ along with a rational equivalence $X \to X_\QQ$. A
rationalization of a simply connected space is unique up to weak
homotopy equivalence and thereby called \emph{the} rationalization.
In a suitable model-category structure in which the weak equivalences
are the rational homotopy equivalences, the rationalization $X_\QQ$ is
a fibrant replacement of $X$, see more on model categories below. We
will present a construction of rationalization in \Cref{rtlztn}.

\section{Model categories}
\label{model-cat}

W.~G. Dwyer \cite{Dwyer.htacs.2008} taught us that ideally a homotopy
theory is determined just by defining the class of weak equivalences
without invoking other ingredients of a model category. Thus, if the
reader prefers to keep the technicalities at bay, they may wish to pay
attention only to the classes of weak equivalences in our model
categories.

A \emph{model-category} structure on a category is characterized by
the choice of three classes of morphisms satisfying certain
properties, see \cite{Dwyer.Spalinski, Hovey}: weak equivalences,
fibrations, and cofibrations. These definitions are modeled upon the
classes of homotopy equivalences, Hurewicz fibrations, and
cofibrations in the category of topological spaces, but this is not
the most commonly used model-category structure on topological spaces
and not the one we will use here. The \emph{standard model category
$\Top$ of topological spaces} selects the following three classes of
morphisms:
\begin{itemize}
\item \emph{weak equivalences}: weak homotopy equivalences;
\item \emph{fibrations}: Serre fibrations;
\item \emph{cofibrations}: continuous maps which have the left lifting
  property (LLP) with respect to Serre fibrations which are weak
  equivalences.
\end{itemize}
Here we use a more general notion of a weak homotopy equivalence than
the one defined earlier in \Cref{RHC} for simply connected spaces:
\begin{quote}
  For general spaces, we say that $f: X \to Y$ is a \emph{weak
  $($homotopy$)$ equivalence} if the induced map $f_*: \pi_n(X, x) \to
  \pi_n(Y, f(x))$ is an isomorphism for each $n \ge 0$.
\end{quote}
For simply connected spaces, this notion is equivalent to the one
which substitutes homology for homotopy, because of the relative
Hurewicz theorem, see the argument in \cite[Corollary 4.33]{Hatcher}.

On our ``intermediary'' category, $\sSet$ of simplicial sets, we
impose a model category structure as follows:
\begin{itemize}
\item \emph{weak equivalences}: simplicial maps whose geometric
  realization is a weak homotopy equivalence;
\item \emph{cofibrations}: simplicial inclusions;
\item \emph{fibrations}: simplicial maps which have the right lifting (RLP)
  property with respect to cofibrations which are weak equivalences.
\end{itemize}
Fibrations in $\sSet$ are known as \emph{Kan fibrations} and may be
characterized by the RLP just with respect to inclusions of horns into
the simplicial representations of simplices, somewhat similar to Serre
fibrations.

Finally, the \emph{model category $\DGCA$} of DGCAs is the category of
(by default, nonnegatively graded and over $\QQ$) DGCAs with
\begin{itemize}
\item quasi-isomorphisms as \emph{weak equivalences};
\item surjective homomorphisms as \emph{fibrations};
\item homomorphisms which have the LLP with respect to fibrations
  which are weak equivalences as \emph{cofibrations}.
\end{itemize}

The axioms of a model category guarantee that every model category
possesses an initial object and a terminal object. An object $Q$ is
called \emph{cofibrant} if the unique morphism from the initial object
to $Q$ is a cofibration. An object $R$ is called \emph{fibrant} if the
unique morphism from $R$ to the final object is a fibration. It is a
fact that space is fibrant, every simplicial set is cofibrant, and
every DGCA is fibrant. A CW complex is an example of a cofibrant
topological space. Fibrant simplicial sets are exactly the \emph{Kan
complexes}; they are characterized by the horn-filling
condition. Sullivan algebras are known to be cofibrant.

The \emph{homotopy category} of a model category $\MM$ is the
localization $\Ho(\MM) \coeq \MM[\WW^{-1}]$ of $\MM$ with respect to
the class $\WW$ of weak equivalences. The localization formally adds
inverses of weak equivalences to the morphisms of $\MM$. One can
develop the notion of \emph{homotopy between morphisms} in a model
category and show that the homotopy category $\Ho(\MM)$ is equivalent
to the category whose objects are fibrant cofibrant objects of $\MM$
and whose morphisms are homotopy classes of morphisms of $\MM$,
cf.\ \cite[Theorem 3.29]{Holstein}.

\begin{thm}[{\cite[Proposition 3.1.5]{Hovey}},
    {\cite[Theorem 17.5.2]{May.Ponto}}]
The functor $\abs{-}$ of geometric realization and the singular
simplex functor $\Sing$ define a Quillen equivalence
  \[
      \xymatrix{
\abs{-} \co \sSet \ar@/^/[r] & \Top : \Sing \ar@/^/[l]
      }.
\]
\end{thm}

The theorem means that these functors establish a Quillen adjunction
and induce an equivalence of the homotopy categories. The
\emph{Quillen adjunction} part means it is an adjunction, i.e., there
is a natural bijection
\[
\Hom_{\Top} (\abs{K}, X) \xrightarrow{\sim} \Hom_{\sSet} (K, \Sing(X))
\]
for simplicial sets $K$ and spaces $X$, and that the functor $\abs{-}$
takes cofibrations to cofibrations and $\Sing$ takes fibrations to
fibrations. A Quillen adjunction always yields the adjunction of the
corresponding (left and right) \emph{derived functors}, in this case
$L\abs{-}$ and $R\Sing$, that is to say, the functors induced on the
homotopy categories:
      \begin{equation}
        \label{sp}
        \xymatrix{
L\abs{-} \co  \Ho(\sSet) \ar@/^/[r] & \Ho(\Top) : R\Sing \ar@/^/[l]
      }.
      \end{equation}
The left derived functor is defined by taking a cofibrant replacement
before applying the left adjoint functor and a fibrant replacement
before applying the right adjoint functor. This works both at the
level of objects and morphisms. In our particular case, since all
simplicial sets are cofibrant and topological spaces are fibrant, we
bypass the replacements at the level of objects: $L\abs{K} = \abs{K}$
and $R\Sing(X) = \Sing (X)$. We still need to replace maps between
simplicial sets by equivalent simplicial inclusions and maps between
spaces by Serre fibrations. All the above replacements exist by the
axioms of model categories.
      
The \emph{Quillen equivalence} part of the theorem statement says that
the functors between the original model categories establish an
equivalence of their homotopy categories. This means that the derived
functors are not only adjoint, but also full and faithful. All in all,
this implies that homotopy theory in $\Top$ is the ``same'' as in
$\sSet$. Choose whichever is more convenient!

\section{The fundamental theorem of RHT}
\label{MainThm}

The fundamental theorem of RHT in a dream world would have said that
the functor $X_{\SR}$ of spatial realization and the polynomial de
Rham algebra functor $A_{\PL}$ establish a Quillen equivalence between
the categories $\Top$ and $\DGCA$. Unfortunately, this does not quite
work, because the spatial realization functor $X_{\SR} =
\abs{K_{\SR}}$ is neither a left, nor a right adjoint functor. There
are also other technicalities which creep in the way, such as the
finiteness and nilpotency conditions. This is another realization, in
addition to the geometric, simplicial, and spatial ones: this time, it
is a sad realization. For this reason, we completely abandon the
category of spaces for the rest of this section and work with the
category of simplicial sets instead.

      \begin{thm}[\cite{Bousfield.Gugenheim} and {\cite[Lemma 6.17]{Holstein}}]
  The functor $K_{\SR}(A)$ of simplicial realization and the
  polynomial de Rham algebra functor $A_{\PL}(K)$ form a Quillen
  adjunction:
  \begin{equation}
    \label{Q-adj}
      \xymatrix{
K_{\SR} \co  \DGCA \ar@/^/[r] & \sSet^{\op} : A_{\PL} \ar@/^/[l]
      }.
      \end{equation}
\end{thm}

Thus, the theorem means that there is a natural bijection
\[
\Hom_{\sSet} (L, K_{\SR}(B)) \xrightarrow{\sim} \Hom_{\DGCA} (B, A_{\PL}(L))
\]
for simplicial sets $L$ and DGCAs $B$ and the functors $K_{\SR}$ and $A_{\PL}$
take cofibrations to fibrations. This implies that the above
adjunction yields the adjunction of the corresponding (left and right)
derived functors $LK_{\SR}$ and $RA_{\PL}$, that is to say, the
functors induced on the homotopy categories:
\[
      \xymatrix{
LK_{\SR} \co  \Ho(\DGCA) \ar@/^/[r] & \Ho(\sSet^{\op}) : RA_{\PL} \ar@/^/[l]
      }.
      \]
Note that since the derived functors are the original functors
precomposed with cofibrant and fibrant replacements, $LK_{\SR} (A) =
K_{\SR}(M(A))$, where $M(A) \to A$ is a minimal Sullivan (or any other
cofibrant) model of a DGCA $A$, and $RA_{\PL}(K) = A_{\PL}(K)$, just
because fibrant objects in the opposite category $\sSet^{\op}$
cofibrant in the original category $\sSet$, in which every object is
cofibrant.
      
The following result ``upgrades'' the Quillen adjunction of the
previous theorem to something resembling a Quillen equivalence and may
be viewed as the fundamental theorem of RHT. However, it holds only
for certain nice simplicial sets and nice DGCAs, which we will
describe now.

Let $\QQ \Nil^{\fft}$ be the full subcategory of $\sSet$ whose objects
are connected, nilpotent, rational Kan complexes of finite
$\QQ$-type. Let $\Ho(\QQ \Nil^{\fft})$ be the category with the same
objects as $\QQ \Nil^{\fft}$ and the morphisms being homotopy classes
of maps. Here \emph{nilpotent} means that the fundamental group is
nilpotent and acts nilpotently on higher homotopy groups. For
instance, all simply connected spaces and simplicial sets are
trivially nilpotent. Other examples include $S^1 = K(\ZZ,1)$ and more
general Eilenberg-MacLane spaces $K(A,n)$ with $A$ abelian, as well as
path-connected based loop spaces. On the other hand, $\RR \PP^\infty$
is not nilpotent, despite an abelian fundamental group. A space or
simplicial set is of \emph{finite $\QQ$-type} if its rational homology
is finite-dimensional in each degree.

On the algebraic side, let $\DGCA^{\cft}$ denote the full subcategory
of $\DGCA$ whose objects are cofibrant, connected differential graded
commutative algebras of finite type. Define $\Ho(\DGCA^{\cft})$ to be
the category whose objects are those of $\DGCA^{\cft}$ and whose
morphisms are homotopy classes of DGCA-homomorphisms. Sullivan
algebras are known to be cofibrant \cite[Theorem
  8.11]{Berglund.RHT}. A DGCA is of \emph{finite type} if it is
finite-dimensional in each degree.

\begin{thm}[\cite{Bousfield.Gugenheim} and {\cite[Theorem 9.5]{Holstein}}]
  \label{Q-eq}
  The Quillen adjunction \eqref{Q-adj} induces an equivalence of categories
      \[
      \xymatrix{ LK_{\SR} \co \Ho(\DGCA^{\cft}) \ar@/^/[r] & \Ho(\QQ
        \Nil^{\fft})^{\op} : RA_{\PL} \ar@/^/[l] }.
      \]
\end{thm}
Here the homotopy categories are the categories with the same objects
but morphisms replaced with homotopy classes of morphisms.

\begin{crl}[{\cite[Corollary 10.2]{Holstein}}, {\cite[Corollary 8.47]{Berglund.RHT}}]
Let $M(K) \xrightarrow{\qis} A_{\PL}(K)$ be a Sullivan minimal model
of a connected nilpotent simplicial set $K$ of finite $\QQ$-type.
Then the adjoint map $K \to K_{\SR} (M(K))$ is a rationalization $K
\to K_\QQ$ of $K$.
\end{crl}

\begin{remark}
  \label{rtlztn}
Similarly, given a path-connected nilpotent CW complex $X$ of
$\QQ$-finite type, take the adjoint $\Sing(X) \to K_{\SR}(M(X))$ of
the Sullivan minimal model $M(X) \xrightarrow{\qis} A_{\PL}(X) =
A_{\PL} (\Sing (X))$. From the derived equivalence \eqref{sp} between
singular sets and spaces, we get a weak equivalence $\abs{\Sing(X)}
\to X$ (the \emph{counit of adjunction}), which is a homotopy
equivalence by Whitehead's theorem. We can compose a homotopy inverse
to get a continuous map $X \to \abs{\Sing (X)} \to X_{\SR}(M(X))$,
which will be a rationalization $X \to X_\QQ$ of $X$.
\end{remark}

\section{Quillen minimal models of spaces}
\label{Quillen-models}

The advantage of Sullivan models over historically earlier Quillen's
models derives from the fact that Sullivan models are more
user-friendly: they are based on more familiar constructions, such as
algebras and differential forms, both known since the late 19th
century. Quillen minimal models use coalgebras and free Lie algebras,
or in modern interpretation, $L_\infty$-algebras, all of which are
less ingrained in the brain of a practitioner and are viewed as less
user-friendly by the general public.

\subsection{Quillen's DGLA model $\lambda X$ of a space $X$}

In the foundational paper \cite{Quillen.1969}, Quillen defines a
sequence of left and right Quillen equivalences, carefully avoiding
calling them that way. At the end of the day, these functors assign to
a simply connected topological space $X$ a differential graded Lie
algebra (DGLA) $\lambda X$,
\[
X \mapsto \lambda X,
\]
defining a functor
\begin{equation}
  \label{lambda}
\lambda: \Top_1 \to \DGLA_0,
\end{equation}
where $\Top_1$ is the category of simply connected spaces and $\DGLA_0$
is the category of \emph{positively graded DGLAs} $L = \bigoplus_{n >
  0} L_n$, with the bracket $[-,-]: L_m \otimes L_n \to L_{m+n}$ being
graded antisymmetric,
\[
  [y,x] = -(-1)^{\abs{x} \cdot \abs{y}} [x,y],
\]
and satisfying a graded version of the Jacobi identity, with a
derivation of degree $-1$:
\[
d[x,y] = [dx,y] + (-1)^{\abs{x}} [x,dy].
\]
Quillen 
constructs $\lambda X$ as the following mouthful:
\[
\lambda X = N \mc{P} \widehat{\QQ} G \Sing (X),
\]
where $G$ associates Kan's loop group to a simplicial set, a
simplicial version of the based loop space, $\widehat{\QQ} G$ is a
completed group algebra, $\mc{P}$ takes the simplicial Lie algebra of
primitive elements, and $N$ maps simplicial Lie algebras to DGLAs by
taking the normalized chains.

Then Quillen uses a similar machinery to that of DGCAs to define the
minimal model of a (positively graded) DGLA and shows that its
isomorphism class determines the rational homotopy type of $X$. Here
we will use the same $\lambda X$ but assign to it a minimal
$L_\infty$-algebra $Q(X)$, rather than a minimal DGLA. We will explain
in which way this is more in line with Sullivan's approach.

\subsection{$\Li$-Algebras}

First, we need a number of algebraic definitions and constructions. A
\emph{$L_\infty$-algebra} is a graded vector space $L = \bigoplus_{n
  \in \ZZ} L_n$ with a codifferential $D$ on the \emph{cofree
coaugemented cocommutative coalgebra}, i.e., the graded symmetric
coalgebra
\[
C(L): = S(L[-1]) = \bigoplus_{n \ge 0} S^n(L[-1]),
\]
where $L[-1] = \bigoplus_{n \ge 2} L[-1]_n$ with $L[-1]_n \coeq
L_{n-1}$ is the suspension of $L$. The comultiplication $\Delta:
S(L[-1]) \to S(L[-1]) \otimes S(L[-1])$ is the standard shuffle
comultiplication, which is coassociative and graded cocommutative. The
coaugmentation is the coalgebra homomorphism $k = S^0(L[-1])
\hookrightarrow S(L[-1])$ and the \emph{codifferential} $D: S(L[-1])
\to S(L[-1])$ is a graded degree-$(-1)$ coderivation,
\[
\Delta (D(x)) = D( x_{(1)}) \otimes x_{(2)} + (-1)^{\abs{x_{(1)}}}
x_{(1)} \otimes D(x_{(2)}),
\]
using \emph{Sweedler's notation} $x_{(1)} \otimes x_{(2)} \coeq \Delta
(x)$, such that $D(1) = 0$ and $D^2 =0$. The codifferential on a
cofree coalgebra is determined by the projection onto the cogenerating
space $L[-1]$:
\[
S(L[-1]) \xrightarrow{D} S(L[-1]) \to L[-1],
\]
and the restriction of this projection to $S^n(L[-1])$ is the
\emph{$n$-ary bracket} $l_n$ of the $L_\infty$-algebra $L$:
\[
l_n: S^n(L[-1]) \hookrightarrow S(L[-1]) \xrightarrow{D} S(L[-1]) \to
L[-1], \qquad n \ge 1.
\]
An $\Li$-algebra is a generalization of a DGLA: a DGLA $(L, [-,-], d)$
may be regarded as an $\Li$-algebra with $l_1 = d$, $l_2 = [-,-]$, and
$l_n = 0$ for $n \ge 0$. The condition $D^2=0$ is equivalent to the
graded Jacobi equation for the bracket $[-,-]$ and the derivation
property of $d$, along with the equation $d^2 = 0$. When higher
brackets $l_n$ are present, the equation $D^2$ gives the Jacobi
identity for $l_2$ up to homotopy given by $l_3$, and similar higher
Jacobi identities. The unary bracket $l_1$ is in fact a linear map $L
\to L$ of degree $-1$ such that $(l_1)^2 = 0$, i.e., a differential on
the underlying graded vector space, denoted by $d \coeq l_1$.

An \emph{$\Li$-homomorphism} $L_1 \to L_2$ is defined as a
homomorphism of the corresponding coaugmented dg-coalgebras $C(L_1)
\to C(L_2)$. An \emph{$\Li$-isomorphism} is an invertible
$\Li$-homomorphism, and an \emph{$\Li$-quasi-isomorphism} is an
$\Li$-homomorphism which induces an isomorphism on homology of the
differentials $d$ on $L_1$ and $L_2$.

By definition, the functor $C$, sometimes called the
\emph{Chevalley-Eilenberg functor}, establishes an equivalence of
categories
\begin{equation}
  \label{CE}
C: \LiAlg_0 \xrightarrow{\sim} \DGCC^{\operatorname{cf}}_1,
\end{equation}
where $\LiAlg_0$ is the category of \emph{positively graded
$\Li$-algebras with $\Li$-morphisms} and $\DGCC^{\operatorname{cf}}_1$
is the category of \emph{cofree coaugmented nonnegatively graded
dg-co\-com\-mu\-ta\-tive coalgebras which are simply connected}, those
with $C_0 = \QQ$ and $C_1 = 0$. The restrictions on grading can be
relaxed to nonnegatively graded $\Li$-algebras and connected
nonnegatively graded DGCCs, or completely removed in an obvious way,
but we will be mostly interested in the equivalence \eqref{CE}.

\subsection{Minimal $\Li$-models}

We say that an $\Li$-algebra is \emph{minimal}, if $d = 0$. We say
that an $\Li$-algebra is \emph{contractible} if $l_n = 0$ for $n \ge
2$ and $H_\bullet (L, d) = 0$ for the homology of $d$. Note that this
is not an $\Li$-invariant notion, but it is invariant under
\emph{linear $\Li$-isomorphisms}, i.e., those which are given by
linear maps of the corresponding coalgebras. An \emph{$\Li$-model} of
an $\Li$-algebra $L$ is an $\Li$-algebra $L'$ with an
$\Li$-quasi-isomorphism $L' \xrightarrow{\qis} L$. As in the theory of
minimal Sullivan models, minimal $\Li$-models play a crucial role.

The theory of minimal $\Li$-models for $\Li$-algebras, developed by
M.~Kontsevich in a couple of impressionistic brush strokes
\cite{Kontsevich.dqopm}, is parallel to the theory of minimal Sullivan
models for DGCAs, but actually simpler.

The following theorem is parallel to \Cref{product} for DGCAs.

\begin{thm}[{\cite[Lemma 4.9]{Kontsevich.dqopm}}, {\cite[Proposition 2.6]{Buijs}}]
  \label{sum}
    Every $\Li$-algebra is $L_\infty$-isomorphic to the direct sum of
    a minimal $\Li$-algebra and a contractible one.
\end{thm}

\begin{proof}
  The proof of this theorem is also simpler than that of its DGCA
  version. If the fact that in Kontsevich's paper
  \cite{Kontsevich.dqopm}, it takes only seven lines is not
  convincing, let us indicate the idea. Since we are working over a
  field, the underlying complex (i.e., the dg-space) $(L,d)$ of an
  $\Li$-algebra $L$ is isomorphic to the direct sum of its homology
  $H_\bullet(L,d)$ and an acyclic complex. The resulting projection of
  the complex $(L,d)$ onto its homology will then be a chain homotopy
  equivalence. Then the homotopy transfer theorem, also known as the
  homological perturbation lemma, which uses the fact that the
  $L_\infty$-operad is cofibrant in a suitable model category of
  operads, see for example, \cite[Theorem 1.23]{Berglund.RHT}, gives an
  $L_\infty$ structure on $H_\bullet(L,d)$ so that the inclusion to
  $L$ and projection onto $H_\bullet(L,d)$ are
  $\Li$-quasi-isomorphisms.
\end{proof}

This theorem and a rather obvious inverse mapping theorem
\cite[Section 4.1]{Kontsevich.dqopm}, which says that an
$L_\infty$-homomorphism whose linear part is an isomorphism of graded
vector spaces is an $L_\infty$-isomorphism, immediately implies the
following analogue of \Cref{unique}.

\begin{thm}[{\cite[Section 4.5.1]{Kontsevich.dqopm}}, {\cite[Theorem 2.7]{Buijs}}]
  \label{minimal}
  \begin{enumerate}[$(1)$]
  \item
    Every $\Li$-quasi-iso\-mor\-phism between minimal $\Li$-algebras
    is an $\Li$-isomorphism.
\item
  Every DGLA has a minimal $\Li$-model, unique up to isomorphism.
  \end{enumerate}
\end{thm}

\subsection{The $\Li$ version of the fundamental  theorem of RHT}

Now back to spaces. Let us use the functor \eqref{lambda} constructed
by Quillen. Given a simply connected space $X$, let us call a minimal
$\Li$-model $Q(X) \xrightarrow{\qis} \lambda X$ of the DGLA $\lambda
X$ the \emph{Quillen minimal model} of $X$. It is defined up to
$\Li$-isomorphism by the theorem above. Note that traditionally, the
Quillen model is defined differently: it is a DGLA which is free as a
graded Lie algebra and with a decomposable differential, see for
example, \cite[Section 1.9]{Berglund.Stoll}. In this way, Quillen's
original approach is also parallel to Sullivan's theory, but the
structure Theorems \ref{product} and \ref{sum} suggest that Koszul
duality between DGCAs and $\Li$-algebras may be a better alternative
to the parallel between semifree DGCAs and semifree DGLAs.

The following analogue of \Cref{Q-eq} works under weaker assumptions,
except for simple connectivity, and deals with covariant, rather than
contravariant functors. Consider the composition of functors
\[
\Top_1 \xrightarrow{\lambda} \DGLA_0 \hookrightarrow \LiAlg_0
\xrightarrow{C} \DGCC_1,
\]
where $\DGCC_1$ is the category of \emph{simply connected coaugmented
nonnegatively graded dg-co\-com\-mu\-ta\-tive coalgebras}. Even though
Quillen does not include in his work the category $\LiAlg_0$, we know
it is equivalent to the full subcategory $\DGCC^{\operatorname{cf}}_1$
of cofree objects in the category $\DGCC_1$.

\begin{thm}[{\cite[Theorem I]{Quillen.1969}}]
  \label{Quillen-eq}
  The functor $\lambda$ composed with the functors $\DGLA_0 \to \LiAlg_0
  \linebreak[0] \to \DGCC_1$ induces equivalences of categories
      \[
\Ho_\QQ(\Top_1) \xrightarrow{\sim} \Ho(\DGLA_0) \xrightarrow{\sim}
\Ho(\LiAlg_0) \xrightarrow{\sim} \Ho(\DGCC_1),
      \]
where $\Ho_\QQ(\Top_1)$ is the localization of the category $\Top_1$ by
rational equivalences and, in the other cases, $\Ho$ stands for the
localization of the corresponding category with respect to
quasi-isomorphisms.
\end{thm}

This theorem implies that the isomorphism class of the Quillen minimal
model $Q(X)$ is a complete invariant of the rational homotopy type of
a simply connected space $X$: such spaces $X$ and $Y$ are rationally
equivalent if and only if their Quillen minimal models $Q(X)$ and
$Q(Y)$ are $\Li$-isomorphic.

One feature of Quillen's functor $\lambda$ is that the rational
homology of the DGLA $\lambda X$ is naturally isomorphic to the
rational homotopy of a simply connected space $X$, with a shift:
\[
H_\bullet(\lambda X, d) \cong \pi_\bullet (X) \otimes \QQ [1].
\]
Therefore, the shifted rational homotopy groups $\pi_\bullet (X)
\otimes \QQ [1]$ acquire a graded Lie bracket. Quillen shows that this
bracket is the Whitehead product \cite[Theorem I]{Quillen.1969}.

By definition, the Quillen minimal model $Q(X)$ is an $\Li$-algebra
based on the graded $\QQ$-vector space
\[
Q(X) \cong H_\bullet (\lambda X, d) \cong \pi_{\bullet} (X) \otimes
\QQ[1],
\]
cf.\ \eqref{cohomology} and \eqref{homotopy} for the Sullivan minimal
model.  The binary bracket of this $\Li$ structure is the Whitehead
product, while the higher $\Li$ brackets have been shown to be the
higher Whitehead products, see \cite{Belchi.Buijs.Moreno}.

\subsection{Relation to Sullivan's theory}

For a simply connected space $X$ of finite $\QQ$-type, the graded
linear dual $C(Q(X))^*$ of the dg-co\-com\-mu\-ta\-tive coalgebra
$C(Q(X))$ of the Quillen minimal model $Q(X)$ is a DGCA. The fact that
the differential on the $\Li$-algebra $Q(X)$ vanishes implies that the
linear part of the differential, the dual of $l_1$, on $C(Q(X))^*$
vanishes, i.e., the differential is decomposable. Also, under the
finiteness assumption, the graded commutative algebra $C(Q(X))^*$ is
isomorphic to $S(Q(X)^*)$. This implies that $C(Q(X))^*$ is a minimal
Sullivan algebra. Thus, a space $X$ of the above type gets two minimal
Sullivan algebras: $C(Q(X))^*$ and the Sullivan minimal model
$M(X)$. Two natural constructions rarely happen to be different. This
is not the case here, either, as the following interpretation of
Majewski's result shows.

\begin{thm}[Majewski {\cite[Theorem 4.90]{Majewski}}, cf.\ {\cite[Remark 2.9]{Buijs}} and {\cite[Remark 1.40]{Berglund.Stoll}}]
  \label{majewski}
  The two minimal Sullivan algebras associated to a simply connected
  space of finite $\QQ$-type, $C(Q(X))^*$ and the Sullivan minimal
  model $M(X)$ are isomorphic.
\end{thm}

\begin{proof}
Under the assumptions, the graded dual $M(X)^* \cong S(\pi_{\bullet}
(X) \otimes \QQ)$, see \eqref{homotopy}, is a dg-co\-com\-mu\-ta\-tive
coalgebra, just like $C(Q(X)) \cong S(\pi_{\bullet} (X) \otimes
\QQ)$. The corresponding $\Li$ structures on $\pi_{\bullet} (X)
\otimes \QQ[1]$ are $\Li$-isomorphic. Indeed, the first one is
$\Li$-quasi-isomorphic to its double bar construction, which is the
bar construction $\LL^*(M(X))$ of $M(X)$ (the free Lie algebra on
$M(X)^*[1]$ with the differential incorporating the DGCA structure on
$M(X)$). By the very definition of $Q(X)$, the second $\Li$ structure
on $\pi_{\bullet} (X) \otimes \QQ[1]$ is $\Li$-quasi-isomorphic to
$\lambda X$, which is quasi-isomorphic as a DGLA to Quillen's original
minimal model $L_X$ of $X$ (the minimal semifree DGLA model of
$\lambda X$). Majewski \cite[Theorem 4.90]{Majewski} proves that the
DGLAs $\LL^*(M(X))$ and $L_X$ are quasi-isomorphic. Therefore, so are
the two $\Li$ structures on $\pi_{\bullet} (X) \otimes \QQ[1]$.  Since
both are minimal, they must be $\Li$-isomorphic by \Cref{minimal}. By
definition of an $\Li$-isomorphism, this means the corresponding
dg-co\-com\-mu\-ta\-tive coalgebras $M(X)^*$ and $C(Q(X))$ are
isomorphic. Thus, the DGCAs $M(X)$ and $C(Q(X))^*$ are isomorphic.
\end{proof}

\subsection{The nerve of an $\Li$-algebra}

Here we will deal with nonnegatively graded $\Li$-algebras for
simplicity. The more general case can be treated by considering
completed $\Li$-algebras, see \cite{Berglund.Stoll}. As these form a
category important for the RHT of simply connected and nilpotent
spaces, it will be general enough for our purposes. Given a
nonegatively graded $\Li$-algebra $L$, let us define the \emph{nerve},
also known as the \emph{Maurer-Cartan space}, see \cite{Hinich.1997,
  Getzler.2009, Berglund.Stoll},
\begin{equation}
  \label{MCE}
\MC (L) := \left\{ x \in (\Omega^\bullet \otimes_\QQ L)_{-1} \mid d x +
\tfrac{1}{2!} l_2(x,x) + \tfrac{1}{3!} l_3(x,x,x) + \dots = 0 \right\},
\end{equation}
where $\Omega^\bullet$ is the simplicial DGCA of \Cref{ratlmodels}
with the opposite grading $\Omega_n \coeq \Omega^{-n}$, so that the de
Rham differential takes degree $-1$. Under the grading assumptions,
the equation \eqref{MCE}, for every element $ x \in
(\Omega^\bullet([n]) \otimes_\QQ L)_{-1}$, will in fact contain at
most finitely many nonzero terms, because the degree of the de Rham
factors of $x$ will have to be strictly negative, and the ideal of
forms of nonzero degree in the polynomial de Rham algebra of the
standard $n$-simplex is nilpotent.

The nerve gives a functor $\MC: \LiAlg_{-1} \to \sSet$, which,
combined with the geometric realization, gives a functor
\[
\LiAlg_{-1} \xrightarrow{\MC} \sSet \xrightarrow{\abs{-}} \Top,
\]
where $\LiAlg_{-1}$ is the category of nonnegatively graded
$\Li$-algebras.  Restricting to positively graded $\Li$-algebras, we
get a functor
\[
\LiAlg_0 \to \Top_1,
\]
which induces an inverse of the equivalence
\[
\Ho_\QQ(\Top_1) \xrightarrow{\sim} \Ho(\LiAlg_0)
\]
of \Cref{Quillen-eq} at least for simply connected spaces of finite
type and positively graded $\Li$-algebras of finite type
(finite-dimensional in each degree), see \cite[Proposition
  6.1]{Berglund.rhtoms.2015}.

As a consequence, we get an isomorphism of graded vector spaces
\begin{equation}
\label{homotopyMC}
\pi_\bullet (\abs{\MC(L)}) \otimes \QQ[1] \cong H_\bullet(L, d)
\end{equation}
for a positively graded $\Li$-algebra $L$ of finite type, see
\cite[Theorem 1.1]{Berglund.rhtoms.2015}. If $L$ is minimal, it is the
Quillen minimal model of $\abs{\MC(L)}$ and its homology coincides
with the algebra, $H_\bullet(L, d) = L$, and acquires the structure of
an $\Li$-algebra. In this case, \eqref{homotopyMC} is an isomorphism
of $\Li$-algebras. This follows from \Cref{majewski}.

\section{Some physical applications}

\subsection{Hypothesis H}

Sati \cite{Sati.2013} made a remarkable observation that the equations
of motion (EOMs) of supergravity in an 11-dimensional ($d=11$)
spacetime are essentially the equations for the differential in the
Sullivan minimal model of $S^4$, see \Cref{even}. Supergravity is the
low-energy, infrared limit of M-theory, which is supposed to be the
``mother of them all:'' by combining S- and T-dualities, one can
arrive at type I, IIA, and IIB string theories as well as various
flavors of heterotic string theories. Dimensional reductions of $d=11$
supergravity lead to supergravity theories in lower dimensions.

The EOMs of $d = 11$ supergravity are the equations
\begin{gather}\label{EOMs}
  d G_4 = 0, \qquad dG_7 = -\frac{1}{2} G_4 \wedge G_4,\\
  \nonumber
  G_7 = * G_4
\end{gather}
on two fields $G_4$ and $G_7$, which are just differential forms on
the 11-dimensional spacetime $X$:
\[
G_4 \in \Omega^4(X), \qquad G_7 \in \Omega^7(X).
\]
Here $*$ is the Hodge star, derived from a given metric on
$X$. Returning to the following Sullivan minimal model of $S^4$:
$M_\RR (S^4) = \RR[g_4, g_7]$ with $\abs{g_4} = 4$, $\abs{g_7} = 7$,
and $d g_4 = 0$, $d g_7 = -\tfrac{1}{2} g_4^2$, we see that there is a
DGCA homomorphism
\begin{align}
  \label{H-map}
  M_\RR (S^4) & \to \Omega^\bullet (X),\\
\nonumber  g_4 & \mapsto G_4,\\
\nonumber  g_7 & \mapsto G_7.
\end{align}
Moreover, the RHT of the four-sphere encompasses the topological part
\eqref{EOMs} of the EOMs, called the \emph{duality-symmetric
EOMs}. The equation $G_7 = *G_4$ is important, but it is geometric,
rather than topological. In a way, the four-sphere carries in its
rational homotopy type information about the duality-symmetric
dynamics of any $d=11$ supergravity theory. RHT gives even more
content to this observation, a continuous map
\begin{equation}
\label{q-cohomotopy}
X \to S^4_\RR
\end{equation}
to the \emph{rationalization of $S^4$ over the reals}, which may be
defined as the spatial realization \linebreak[0] $X_{\SR}
(M_\RR(S^4))$. ($S^4_\RR$ has the same real homotopy type as $S^4$ via
a real homotopy equivalence $S^4 \to S^4_\RR$.) The above map is
obtained from the adjoint of \eqref{H-map}, see \Cref{rtlztn}. It
intertwines the fields $G_4$ and $G_7$ on spacetime $X$ with the
generators $g_4$ and $g_7$ of the Sullivan minimal model of
$S^4$. Thus, $S^4_\RR$ becomes the universal target space of $d=11$
supergravity in any spacetime. These observations is the essence of
Sati's \emph{Hypothesis H} \cite{Sati.2013}.

Supergravity in 11 dimensions may be reduced to lower dimensions, and
Hypothesis H extends to establish similar ties between the EOMs and
the Sullivan minimal models of certain loop spaces of $S^4$, along
with similar maps.

The $(11-k)$-dimensional reduction of 11-dimensional supergravity is
obtained under the assumption that there is a $k$-torus $(S^1)^k$
acting on spacetime $Y$ of $\dim Y = 11$, so that the new spacetime is
the quotient $Y/(S^1)^k$. Let us look at the case $k =1$, as this case
is easily generalized to higher $k$'s. There is an adjunction
\begin{equation}
\label{adjunction}
\Hom_{/BS^1} (Y / S^1, \mc{L}_c Z) \xrightarrow{\;\; \sim \;\;}
\Hom (Y, Z),
\end{equation}
where $Y$ is a space with a free action of $S^1$, $Z$ is another
topological space, $\mc{L}_c Z \coeq \Map(S^1, Z) \dslash S^1$ is the
\emph{cyclic $($or equivariant$)$ loop space of $Z$}, the left-hand
side is the set of morphisms in the category $/BS^1$ of spaces over
$BS^1$ (equivalent to the category of principal $S^1$-bundles, such as
$Y \to Y/S^1$ and $\mc{L} Z \to \mc{L}_c Z$), and the right-hand side
is the set of continuous maps $Y \to Z$; see
\cite{FSS-L00}\cite[Theorem 2.44]{BSS}. This adjunction produces a map
\begin{equation}
  \label{cyclified}
Y/S^1 \longrightarrow \mc{L}_c S^4_\RR
\end{equation}
from the map \eqref{q-cohomotopy}. If the action of $S^1$ on $Y$ is
not free, one has to replace the naive quotient $Y/S^1$ by the
homotopy quotient $Y \dslash S^1$ in the above.  Roughly speaking,
thinking of $\mc{L}_c S^4_\RR$ as the space of unparameterized (or
equivariant) free loops in $S^4_\RR$, a map like \eqref{q-cohomotopy}
from an $S^1$-space $Y$ will produce a map \eqref{cyclified} which
assigns to a point $y \in Y/S^1$ the map that takes the $S^1$-orbit in
$Y$ over $y$ to $S^4_\RR$ by the given map $Y \to S^4\RR$.

10-Dimensional supergravity (corresponding to type IIA string theory),
is based on the spacetime $Y/S^1$ and, apart from the above map
\eqref{cyclified}, it is related to the rational homotopy type of
$\mc{L}_c S^4_\RR$ via the (duality-symmetric) EOMs, which explicitly
match the equations for the differential of the Sullivan minimal model
$M(\mc{L}_c S^4)$:
\begin{gather*}
  M(\mc{L}_c S^4) = (\RR[g_4, g_7, sg_4, sg_7, w], d)\;,\\
  dg_4 = sg_4 \cdot w, \qquad dg_7 = -\tfrac{1}{2} g_4^2 + sg_7 \cdot w\;,\\
  d sg_4 = 0, \quad dsg_7 = sg_4 \cdot g_4, \quad d w = 0\;,
\end{gather*}
where $\deg w =2$ and, for each generator $x$ of $M(S^4)$, $sx$ is a
new generator of $\deg s x = \deg x -1$. The Sullivan minimal model $
M(\mc{L}_c Z)$ of the cyclic loop space of a path-connected, nilpotent
space $Z$ is obtained via a simple algebraic procedure developed in
the good old times of RHT, a theorem of Vigu\'e-Poirrier and Burghelea
\cite{vigue-burghelea}.  This matching of equations of $d=10$, type
IIA supergravity with the Sullivan minimal model of $M(\mc{L}_c S^4)$
was observed and developed by D.~Fiorenza, U.~Schrei\-ber, and Sati
\cite{FSS15,FSS17,FSS-pbranes} and provided more evidence for
Hypothesis H: it is compatible with reduction to 10 dimensions.

The papers \cite{SV1, SV2} extended Hypothesis H to further
dimensional reductions of supergravity. Name\-ly, the same pattern
persists for the values of $k$ through $k = 10$, which is a matter of
explicit computation of $M(\mc{L}_c^k S^4)$ by iterating
Vigu\'e-Poirrier and Burghelea's theorem and comparing it to the known
EOMs of dimensionally reduced supergravity theories. We may summarize
this application of RHT to Mathematical Physics as follows:
\begin{quote}
 {\it The iterated cyclic loop space $\mc{L}_c^k S^4_\RR$ admits a map
   from the $(11-k)$-dimensional supergravity spacetime, thus making
   $\mc{L}_c^k S^4_\RR$ a universal target for all possible spacetimes
   $Y/(S^1)^k$. The Sullivan minimal model $M(\mc{L}_c^k S^4)$
   determines the duality-symmetric equations of motion of
   $(11-k)$-di\-men\-sion\-al supergravity reduced from 11-dimensional
   supergravity.}
\end{quote}

This matching also implies the following general principle:
\begin{quote}
{\it Any feature of or statement about the Sullivan minimal model of
  an iterated cyclic loop space $\mc{L}_c^k S^4$ (or the rational
  homotopy type thereof) may be translated into a feature of or
  statement about the compactification of M-theory on the $k$-torus.}
\end{quote}

\subsection{Mysterious Triality and cyclic loop spaces of $S^4$}

Mysterious Duality has been tantalizing the world of mathematical
physics since 2001, when A.~Iqbal, A.~Neitzke, and C.~Vafa
\cite{iqb-n-v} discovered a convincing, yet mysterious correspondence
between certain symmetry patterns in the geometry of del Pezzo
surfaces, see \cite{manin:cubic}, and in toroidal compactifications of
M-theory, both governed by the root system series $E_k$.

In the context of algebraic geometry, the $E_k$ root system appears
out of intersection theory on del Pezzo surfaces by a simple algebraic
argument: given a \emph{Lorentzian}, \emph{i.e}., of index $(1,k)$,
lattice of rank $k+1$ with $k \le 8$, let $-K$ be the vector $ -K =
(3, -1, -1, \dots, -1)$ with respect to a Lorentzian basis; then the
set of lattice vectors of inner square $-2$ and orthogonal to $-K$
form a root system that may be shown to be the root system of type
$E_k$. What is remarkable in the geometry of del Pezzo surfaces is
that all this data is natural: the lattice is the second cohomology
group $H^2 (\BB_k; \ZZ)$ of the del Pezzo surface $\BB_k$, the
Lorentzian inner product is the intersection form, and $-K$ is the
anticanonical class, the first Chern class of the holomorphic tangent
bundle of $\BB_k$.


In \cite{SV1}, it was observed that the same data (a Lorentzian
lattice and distinguished vector) as for del Pezzo surfaces above
naturally appears in the rational homotopy theory models for the
iterated cyclic loop spaces of $4$-spheres $S^4$:
\[
S^4, \mc{L}_c S^4, \mc{L}_c^2 S^4, \dots, \mc{L}_c^8 S^4.
\]
We assume that if $\mc{L}_c Z$ on this list happens not to be
connected, we keep only the connected component of the constant loop
$S^1 \to Z$ and slightly abuse notation by using the same symbol
$\mc{L}_c Z$ for the connected component.

Here is how the $E_k$ root system shows up: it is naturally associated
to the infinitesimal symmetries of the Sullivan minimal model
$M(\mc{L}_c^k S^4)$.  The paper \cite{SV1} shows that that a maximal
$\QQ$-split torus $T$ in the automorphism group $\Aut M(\mc{L}_c^k
S^4)$ of $M(\mc{L}_c^k S^4)$ is isomorphic to $(\QQ^\times)^{k+1}$,
where $\QQ^\times := \QQ \setminus\{0\} = \GG_m(\QQ)$ is the
\emph{multiplicative group} of the rationals, and so the Lie algebra
$\h = \Lie (T)$ of this maximal split torus $T$ is a $(k +
1)$-dimensional rational vector space. The Lie algebra $\h$ is
canonically equipped with a Lorentzian metric and a distinguished
Lorentzian basis. Indeed, this data is equivalent to the data of two
distinguished bases, one for $\h$ and one for the dual vector space
$\h^*$. For the dual space $\h^*$, a canonical basis comes from the
weight lattice for the $\h$-action on $M (\mc{L}_c^k S^4)$.
A distinguished basis of $\h$ comes directly from the rational
homotopy theory of $\mc{L}_c^k S^4$. In the rational homotopy
category, both the degree-$n$ folding maps of the source $S^1$’s and
of the target $S^4$ for $\mc{L}_c^k S^4$ are invertible as long as $n
\ne 0$. Taking into account their inverses, one has a natural action
(in the rational homotopy category) of $k + 1$ copies of $\QQ^\times$
on $\mc{L}_c^k S^4$, one for each of the source circles $S^1$ and one
for the target $S^4$.
And since the action of the folding maps on the loop space is by
precomposition in the sources and postcomposition in the target, one
gets from this that the eigenvalues for the infinitesimal generators of
the folding actions are $-1$ for the infinitesimal source foldings and
$+1$ for the target folding, thus giving a natural Minkowski inner product.
Finally, the distinguished vector $-K$ is also quite natural:
it is a unique element $-K \in \h$ which acts as the degree operator
on the Quillen minimal model $\QQ(\mc{L}_c^k S^4)$ of $\mc{L}_c^k
S^4$. Here, the Quillen minimal model is understood in the sense of
\Cref{Quillen-models} as a minimal $\Li$ structure on $\pi_\bullet
(\mc{L}_c^k S^4) \otimes \QQ[1]$. Because the differential of the
Sullivan minimal model $M(\mc{L}_c^k S^4)$ happens to be purely
quadratic, the Quillen minimal model $Q(\mc{L}_c^k S^4)$ is just a
graded Lie algebra, being defined by the Sullivan minimal model
$M(\mc{L}_c^k S^4)$ via the Sullivan/Quillen-model duality
\Cref{majewski}.

The above sketchy argument may be summarized in the following theorem.

 \begin{thm}
 \begin{enumerate}[{\bf (a)}]
 \setlength\itemsep{-2pt}
 \item   
The maximal $\QQ$-split torus of the rational algebraic group $\Aut M
(\mc{L}_c^k S^4)$ for $k \ge 0$ is a $(k+1)$-dimensional torus
$T^{k+1}$ canonically isomorphic to $\GG_m^{k+1}$ over $\QQ$.

\item 
The action of the maximal split torus $\GG_m^{k+1}$ on $M(\mc{L}^k_c
S^4)$ may be lifted to an action on the space $\mc{L}_c^k S^4$ in the
rational homotopy category. In this way, the last factor $\GG_m$ of
$\GG_m^{k+1}$ acts via self-maps of the target $S^4$ and the first $k$
factors act via self-maps of the source $S^1$'s.

\item
\label{bases}
The $(k+1)$-dimensional rational abelian Lie algebra $ \h_k = \Lie
(T^{k+1}) \subseteq \linebreak[1] \Der \linebreak[0] M(\mc{L}_c^k
S^4), $ which plays the role of a Cartan subalgebra, of the maximal
$\QQ$-split torus $T = T^{k+1}$ of the rational algebraic group $\Aut
M(\mc{L}_c^k S^4)$ has an explicit canonical basis. So does the linear
dual $\h_k^*$, which plays the role of a weight space.

\item
\label{element}
There is a unique element $-K$ of the Lie algebra $\h_k$, which acts
on the Quillen minimal model $ Q(\mc{L}_c^k S^4)$ as the degree
operator.

\item For each $k$, $0 \le k \le 8$, the basis of $\h_k$ from {\rm
  \ref{bases}} and element $-K$ from {\rm \ref{element}} above form a
  Lorentzian lattice of rank $k+1$ with a distinguished element,
  giving rise to the exceptional root system $E_k$. This data,
  extracted from the cyclic loop space $M(\mc{L}^{k}_cS^4)$,
  replicates the root data determined by the del Pezzo surface
  $\BB_k$.  The construction of the root data
 extends
 to $k\geq 9$ to the root systems of Kac-Moody algebras of type $E_k$.

\item \textbf{$27$ Lines via rational homotopy of $6$-fold cyclic loop
  space}: In the weight-space decomposition
\[
\pi_\bullet (\mc{L}_c^6 S^4)  \otimes \QQ = \bigoplus_{\alpha \in \h_6^*}
\left(\pi_\bullet (\mc{L}_c^6 S^4) \otimes \QQ \right)_\alpha
\]

\vspace{-3mm} 
\noindent
corresponding to the $7$-torus action on the Quillen minimal model
$Q(\mc{L}_c^6 S^4) \linebreak[0] = \pi_\bullet(\mc{L}_c^6 S^4) \otimes
\QQ[1]$, the $27$ exceptional vectors $\alpha_i \in \h_6^*$, $i = 1,
\dots, 27$,
single out precisely the second rational homotopy group $\pi_2
(\mc{L}_c^6 S^4)  \otimes \QQ$:
\[
\pi_2 (\mc{L}_c^6 S^4) \otimes \QQ = \bigoplus_{i=1}^{27} \left(
\pi_\bullet (\mc{L}_c^6 S^4) \otimes \QQ \right)_{\alpha_i} \; .
\]
Moreover,
\[
\dim \left(\pi_\bullet (\mc{L}_c^6 S^4) \otimes \QQ \right)_{\alpha_i}
= 1
\]
for each $i = 1, \dots, 27$, which means there are $27$ canonically
defined, linearly independent lines in the $\QQ$-vector space
$\pi_2(\mc{L}_c^6 S^4) \otimes \QQ$ and $\dim \pi_2 (\mc{L}_c^6 S^4)
\otimes \QQ = 27$.
\end{enumerate} 

\end{thm}

 \subsection{Other applications to physics}
 
Sati and M.~Wheeler \cite{SW18} studied rationalizations of string,
fivebrain, and ninebrain structures on the tangent space of a
manifold. These are liftings of the standard structure group $O(n)$ of
the tangent bundle of an $n$-manifold to the groups
$\operatorname{String}(n)$, $\operatorname{Fivebrain}(n)$, and
$\operatorname{Ninebrain}(n)$, respectively, which are, in a certain
way, higher generalizations of the spin group
$\operatorname{Spin}(n)$. Rationalization allowed for considerable
simplification of these structures, and RHT was applied to describe
the space of these rationalized structures. Sati and Wheeler
\cite{SW19} used these results to describe the topological parts of
action functionals in string and M-theories.

Surveys of some other applications of RHT to Mathematical Physics may
be found in \cite{FSS.trhsomt} and \cite{sati2024flux}.

\bibliographystyle{amsalpha}
\bibliography{RHT}


\end{document}